\documentclass[11pt,a4paper, final, twoside]{article}
\usepackage{amsmath}
\usepackage{cite}
\usepackage{fancyhdr}
\usepackage{amsthm}
\usepackage{amsfonts}
\usepackage{amssymb}
\usepackage{amscd}
\usepackage{graphicx}
\usepackage{afterpage}
\usepackage[colorlinks=true, urlcolor=blue,  linkcolor=blue, citecolor=blue]{hyperref}
\newtheorem{theorem}{Theorem}[section]

\newtheorem{definition}[theorem]{Definition}
\newtheorem{example}[theorem]{Example}

\newtheorem{lemma}[theorem]{Lemma}

\newtheorem{remark}[theorem]{Remark}

\begin{document}
\begin{center}
{\Large \bf{$C$-class functions on some fixed point results  of integral type 
and applications}}
\end{center} \vspace{10mm}
\centerline{Arsalan Hojat Ansari$^{1}$, Bahman Moeini$^{2}$ and Seyed.M.A. Aleomraninejad$^{3,*}$}\ \\
\centerline{\footnotesize  $^{1}$Department of Mathematics, Karaj Branch, Islamic Azad University, Karaj, Iran}
\centerline{\footnotesize  $^{2}$Department of Mathematics, Hidaj Branch, Islamic Azad University, Hidaj, Iran}
\centerline{\footnotesize  $^{3}$Department of Mathematics, Qom University of Technology, Qom, Iran}
\vspace{1mm}
\footnote{$^*$ Corresponding author}
\footnote{E-mail:mathanalsisamir4@gmail.com}
\footnote{E-mail:moeini145523@gmail.com}
\footnote{E-mail:aleomran63@yahoo.com}
\footnote{\emph{}} \afterpage{} \fancyhead{} \fancyfoot{}
\fancyhead[LE, RO]{\bf\thepage} \fancyhead[LO]{\small $C$-class functions on 
some fixed point results  of integral type and applications} \fancyhead[RE]{\small
A.H. Ansari, B. Moeini and S.M.A. Aleomraninejad }
\begin{abstract}
In this paper, by using $C$-class functions \cite{16} for integral type of Suzuki-type mappings, some 
fixed point results are established on a metric space that generalize the results of Aleomraninejad and Shokouhnia 
[Adv. Fixed Point Theory, 5 (2015), No. 1, 101-109]. As an application, the existence of a continuous solution for an integral equation is obtained.
\end{abstract}
\section{\textbf{Introduction }}

The first important result on fixed points for contractive-type mappings was
the well-known Banach contraction theorem, published for the first time in
1922 (\cite{4}). In the general setting of complete metric spaces, this theorem
runs as follows.

\begin{theorem}
\label{1.1} Let $(X,d)$ be a complete metric space, $\beta \in (0,1)$ and
let $T:X\rightarrow X$ be a mapping such that for each $x,y\in X,$ 
\begin{equation*}
d(Tx,Ty)\leq \beta d(x,y).
\end{equation*}%
Then $T$ has a unique fixed point $a\in X$ such that for each $x\in X$, $%
\lim_{n\rightarrow \infty }T^{n}x=a$.
\end{theorem}

In order to generalize this theorem, many authors have introduced various
types of contraction inequalities. In 2002, Branciari proved the following
result (see \cite{5}).

\begin{theorem}
\label{1.2} Let $(X,d)$ be a complete metric space, $\beta \in (0,1)$ and $%
T:X\longrightarrow X$ a mapping such that for each $x,y\in X$, 
\begin{equation*}
\int_{_{0}}^{d(Tx,Ty)}f(t)dt\leq \beta \int_{_{0}}^{d(x,y)}f(t)dt,
\end{equation*}%
where $f:[0,\infty )\rightarrow (0,\infty )$ is a Lebesgue integrable
mapping which is summable (i.e., with finite integral on each compact subset
of $[0,\infty )$) and for each $\varepsilon >0$, $\int_{_{0}}^{\varepsilon
}f(t)dt>0.$ Then $T$ has a unique fixed point $a\in X$ such that for each $%
x\in X$, $\lim_{n\rightarrow \infty }T^{n}x=a$.
\end{theorem}

In 2008, Suzuki introduced a new method in \cite{13} and then his method was
extended by some authors (see for example, \cite{7,8,9,13,14}). Kikkawa and
Suzuki extended the method in \cite{8} and then Mot and Petru\c{s}el further
generalized it in \cite{9}. The following theorem is the result of Theorem 2.2 in
\cite{2}.

\begin{theorem}
\label{1.3} Let $(X,d)$ be a complete metric space and $T:X\longrightarrow X$
a mapping. Suppose that there exist $\alpha \in (0,\frac{1}{2}]$, $\beta \in
(0,1)$ such that $\alpha d(x,Tx)\leq d(x,y)$ implies $d(Tx,Ty)\leq \beta
d(x,y)$ for all $x,y\in X$. Then $T$ has a unique fixed point $a\in X$ such
that for each $x\in X$, $\lim_{n\rightarrow \infty }T^{n}x=a$.
\end{theorem}

ّIn 2015, Aleomraninejad and Shokouhnia \cite{0} by idea of Suzuki and Branciari
established  the following theorem. 
\begin{theorem}
\label{1.4} Let $(X,d)$ be a complete metric space and $T:X\longrightarrow X$
a mapping. Suppose that there exist $\alpha \in (0,\frac{1}{2}]$, $\beta \in
(0,1)$ such that $%
\alpha d(x,Tx)\leq d(x,y)$ implies  
\begin{equation*}
\int_{_{0}}^{d(Tx,Ty)}f(t)dt\leq \beta \int_{_{0}}^{d(x,y)}f(t)dt,
\end{equation*}%
for all $x,y\in X$ and $f:[0,\infty )\rightarrow (0,\infty )$ is a Lebesgue integrable
mapping which is summable (i.e., with finite integral on each compact subset
of $[0,\infty )$) and for each $\varepsilon >0$, $\int_{_{0}}^{\varepsilon
}f(t)dt>0.$ Then $T$ has a unique fixed point $a\in X$ such that for each $%
x\in X$, $\lim_{n\rightarrow \infty }T^{n}x=a$.
\end{theorem}

The aim of this paper is to use of $C$-class functions and provide a new condition for  integral type mapping $T$ which guarantees the existence of its fixed point in a metric space by idea of Aleomraninejad and Shokouhnia. Our
results generalize some old results. In this way, we appeal the following
notions.
\section{Basic notions}

\begin{lemma}
\label{1.6} Let $a,b\in \lbrack 0,\infty )$ and $f:[0,\infty )\rightarrow
(0,\infty )$ a Lebesgue integrable mapping which is summable and for each $%
\varepsilon >0$, $\int_{_{0}}^{\varepsilon }f(t)dt>0$. Then \newline
i) $a=0$ whenever $\int_{_{0}}^{a}f(t)dt=0$,\newline
ii) $a<b$ whenever $\int_{_{0}}^{a}f(t)dt<\int_{_{0}}^{b}f(t)dt$.
\end{lemma}

\begin{lemma}
Let $L>0$, $\alpha (x),\beta (x)\in C([a,b])$ and $f:[0,\infty )\rightarrow
(0,\infty )$ a Lebesgue integrable mapping which is summable and for each $%
\varepsilon >0$, $\int_{_{0}}^{\varepsilon }f(t)dt>0$. Then $%
\int_{_{0}}^{\Vert \alpha \Vert _{\infty }}f(t)dt<L\int_{_{0}}^{\Vert \beta
\Vert _{\infty }}f(t)dt$ whenever $\int_{_{0}}^{\vert\alpha
(x)\vert}f(t)dt<L\int_{_{0}}^{\vert\beta (x)\vert}f(t)dt$.
\end{lemma}

in 2014 A.H. Ansari \cite{16} introduced the concept of $C$-class functions which
cover a large class of contractive conditions.

\begin{definition}
\label{C-class} \cite{16} A continuous function $F:[0,\infty )^{2}\rightarrow 
\mathbb{R}$ is called \textit{$C$-class } function if \ for any $s,t\in
\lbrack 0,\infty ),$ the following conditions hold:

(1) $F(s,t)\leq s$;

(2) $F(s,t)=s$ implies that either $s=0$ or $t=0$.
\end{definition}

An extra condition on $F$ that $F(0,0)=0$ could be imposed in some cases if
required. The letter $\mathcal{C}$ will denote the class of all $C$-
functions.

\begin{example}
\label{C-class examp}\cite{16} Following examples show that the class $\mathcal{C}
$ is nonempty:
\end{example}

\begin{enumerate}
\item $F(s,t)=s-t.$

\item $F(s,t)=ms,$for some $m\in (0,1).$

\item $f(s,t)=\frac{s}{(1+t)^{r}}$ for some $r\in (0,\infty ).$

\item $f(s,t)=\log (t+a^{s})/(1+t)$, for some $a>1.$

\item $f(s,t)=\ln (1+a^{s})/2$, for $e>a>1.$ Indeed $f(s,t)=s$ \ implies
that $s=0.$

\item $f(s,t)=(s+l)^{(1/(1+t)^{r})}-l$, $l>1,$ for $r\in (0,\infty )$.

\item $f(s,t)=s\log _{t+a}a$, for $a>1$.

\item $f(s,t)=s-(\frac{1+s}{2+s})(\frac{t}{1+t}).$

\item $f(s,t)=s\beta (s)$, where $\beta :[0,\infty )\rightarrow \lbrack
0,1). $and continuous

\item $f(s,t)=s-\frac{t}{k+t}.$

\item $f(s,t)=s-\varphi (s),$ where $\varphi :[0,\infty )\rightarrow \lbrack
0,\infty )$ is a continuous function such that $\varphi (t)=0$ if and only
if $t=0.$

\item $f(s,t)=sh(s,t),$ where $h:[0,\infty )\times \lbrack 0,\infty
)\rightarrow \lbrack 0,\infty )$ is a continuous function such that $%
h(t,s)<1 $ for all $t,s>0$.

\item $f(s,t)=s-(\frac{2+t}{1+t})t.$

\item $f(s,t)=\sqrt[n]{\ln (1+s^{n})}.$

\item $F(s,t)=\phi (s),$ where $\phi :[0,\infty )\rightarrow \lbrack
0,\infty )$ is a upper semicontinuous function such that $\phi (0)=0$ and $%
\phi (t)<t$ for $t>0.$

\item $F(s,t)=\frac{s}{(1+s)^{r}}$; $r\in (0,\infty )$.

\item $F(s,t)=\frac{s}{\Gamma (1/2)}\int_{0}^{\infty }\frac{e^{-x}}{\sqrt{x}%
+t}\,dx$, where $\Gamma $ is the Euler Gamma function.
\end{enumerate}

Let $\Phi _{u}$ denote the class of the functions $\varphi :[0,\infty
)\rightarrow \lbrack 0,\infty )$ which satisfy the following conditions:

\begin{enumerate}
\item[$(\varphi_{1})$] $\varphi $ continuous ;

\item[$(\varphi_{2})$] $\varphi (t)>0,t>0$ \ and $\varphi (0)\geq 0$\ .
\end{enumerate}

Let $\Psi $ be a set of all continuous functions $\psi :[0,\infty)
\rightarrow [0,\infty)$ satisfying the following conditions:

\begin{itemize}
\item[$(\protect\psi_{1})$] $\psi$ is continuous and strictly increasing.

\item[$(\protect\psi _{2})$] $\psi (t)=0$ if and only of $t=0$.
\end{itemize}

\begin{lemma}
\label{goodlem} \cite{17} Suppose $(X,d)$ is a metric space. Let $\{x_{n}\}$ be a
sequence in $X$such that $d(x_{n},x_{n+1})\rightarrow 0$ as $n\rightarrow
\infty $. If $\{x_{n}\}$ is not a Cauchy sequence then there exists an $%
\varepsilon >0$ and sequences of positive integers $\{m(k)\}$ and $\{n(k)\}$
with
\end{lemma}
$m(k)>n(k)>k$ such that $d(x_{m(k)},x_{n(k)})\geq \varepsilon $, $%
d(x_{m(k)-1},x_{n(k)})<\varepsilon $ and\newline
(i) $\lim_{k\rightarrow \infty }d(x_{m(k)-1},x_{n(k)+1})=\varepsilon $;\newline
(ii) $\lim_{k\rightarrow \infty }d(x_{m(k)},x_{n(k)})=\varepsilon$;\newline
(iii) $\lim_{k\rightarrow \infty }d(x_{m(k)-1},x_{n(k)})=\varepsilon$.\newline
\\
We note that also can see 
\[\lim_{k\rightarrow \infty
}d(x_{m(k)+1},x_{n(k)+1})=\varepsilon  \ \ \text{and}\ \ \lim_{k\rightarrow \infty
}d(x_{m(k)},x_{n(k)-1})=\varepsilon.
\]
\section{Main Results}

The following theorem is the main result of this paper.

\begin{theorem}
\label{2.1} Let $(X,d)$ be a complete metric space and $T:X\longrightarrow X$
a mapping. Suppose that there exists $\alpha \in (0,\frac{1}{2}]$, such that $%
\alpha d(x,Tx)\leq d(x,y)$ implies 
\begin{align}
\psi (\int_{_{0}}^{d(Tx,Ty)}f(t)dt)\leq F(\psi
(\int_{_{0}}^{d(x,y)}f(t)dt),\varphi (\int_{_{0}}^{d(x,y)}f(t)dt)),
\label{1}
\end{align}
for all $x,y\in X$ ,\ $\psi \in \Psi ,\varphi \in \Phi _{u},F\in \mathcal{C}$ and 
$f:[0,\infty )\rightarrow (0,\infty )$ is a Lebesgue integrable mapping
which is summable and for each $\varepsilon >0$, $\int_{_{0}}^{\varepsilon
}f(t)dt>0.$ Then $T$ has a unique fixed point $a\in X$ such that for each $%
x\in X$, $\lim_{n\rightarrow \infty }T^{n}x=a$.
\end{theorem}

\begin{proof}
Fix arbitrary $x_{0}\in X$ and $x_{1}=Tx_{0}$. We have $\alpha
d(x_{0},Tx_{0})<d(x_{0},x_{1})$. Hence, by (\ref{1})
\begin{eqnarray*}
\psi (\int_{_{0}}^{d(Tx_{0},Tx_{1})}f(t)dt) &\leq &F(\psi
(\int_{_{0}}^{d(x_{0},x_{1})}f(t)dt),\varphi
(\int_{_{0}}^{d(x_{0},x_{1})}f(t)dt)) \\
&\leq &\psi (\int_{_{0}}^{d(x_{0},x_{1})}f(t)dt).
\end{eqnarray*}%
Since $\psi \in \Psi $, we have $\int_{_{0}}^{d(x_{1},Tx_{1})}f(t)dt<%
\int_{_{0}}^{d(x_{0},x_{1})}f(t)dt$. Let $x_{2}=Tx_{1}$. By Lemma \ref{1.6}, 
$d(x_{1},Tx_{1})<d(x_{0},x_{1})$, so $\alpha d(x_{1},Tx_{1})<d(x_{1},x_{2})$
and 
\begin{eqnarray*}
\psi (\int_{_{0}}^{d(Tx_{1},Tx_{2})}f(t)dt) &\leq &F(\psi
(\int_{_{0}}^{d(x_{1},x_{2})}f(t)dt),\varphi
(\int_{_{0}}^{d(x_{1},x_{2})}f(t)dt)) \\
&\leq &\psi (\int_{_{0}}^{d(x_{1},x_{2})}f(t)dt).
\end{eqnarray*}%
Now let $x_{3}=Tx_{2}$. By Lemma \ref{1.6}, $%
d(x_{2},x_{3})<d(x_{1},x_{2})<d(x_{0},x_{1})$. Since $\alpha
d(x_{2},Tx_{2})<d(x_{2},x_{3})$, 
\begin{eqnarray*}
\psi (\int_{_{0}}^{d(Tx_{2},Tx_{3})}f(t)dt) &\leq &F(\psi
(\int_{_{0}}^{d(x_{2},x_{3})}f(t)dt),\varphi
(\int_{_{0}}^{d(x_{2},x_{3})}f(t)dt)) \\
&\leq &\psi (\int_{_{0}}^{d(x_{2},x_{3})}f(t)dt).
\end{eqnarray*}%
By continuing this process, we obtain a sequence $\{x_{n}\}_{n\geq 1}$ in $X$
such that $x_{n+1}=Tx_{n},~d(x_{n},x_{n+1})<d(x_{n-1},x_{n})$ and 
\begin{equation*}
\psi (\int_{_{0}}^{d(x_{n},x_{n+1})}f(t)dt)\leq F(\psi
(\int_{_{0}}^{d(x_{n-1},x_{n})}f(t)dt),\varphi
(\int_{_{0}}^{d(x_{n-1},x_{n})}f(t)dt)).
\end{equation*}%
We claim that for any $y\in X$, one of the following relations is hold: 
\begin{align}
\label{2}
\alpha d(x_{n},Tx_{n})\leq d(x_{n},y)\  \ \text{or} \ \ \alpha d(x_{n+1},Tx_{n+1})\leq
d(x_{n+1},y).
\end{align}
Otherwise, if $\alpha d(x_{n},Tx_{n})>d(x_{n},y)$ and $\alpha
d(x_{n+1},Tx_{n+1})>d(x_{n+1},y)$, we have 
\begin{equation*}
d(x_{n},x_{n+1})\leq d(x_{n},y)+d(x_{n+1},y)<\alpha d(x_{n},Tx_{n})+\alpha
d(x_{n+1},Tx_{n+1})
\end{equation*}%
\begin{equation*}
=\alpha d(x_{n},x_{n+1})+\alpha d(x_{n+1},x_{n+2})\leq 2\alpha
d(x_{n},x_{n+1})\leq d(x_{n},x_{n+1}),
\end{equation*}%
which is a contradiction. Now let $a_{n}=d(x_{n},x_{n+1})$ for all $n\geq 1$%
. It is obvious that $\{a_{n}\}_{n\geq 1}$ is monotone non-increasing and so
there exists $a\geq 0$ such that $\lim_{n\rightarrow \infty }a_{n}=a$. Since 
\begin{eqnarray*}
\psi (\int_{_{0}}^{a}f(t)dt) &=&\psi (\lim_{n\rightarrow \infty
}\int_{_{0}}^{a_{n}}f(t)dt) \\
&\leq &F(\lim_{n\rightarrow \infty }\psi
(\int_{_{0}}^{d(x_{n-1},x_{n})}f(t)dt),\lim_{n\rightarrow \infty }\varphi
(\int_{_{0}}^{d(x_{n-1},x_{n})}f(t)dt)) \\
&=&F(\psi (\int_{_{0}}^{a}f(t)dt),\varphi (\int_{_{0}}^{a}f(t)dt))\\
&\leq& \psi (\int_{_{0}}^{a}f(t)dt).
\end{eqnarray*}%
So, $\psi (\int_{_{0}}^{a}f(t)dt)=0$ or $\varphi(\int_{_{0}}^{a}f(t)dt)=0$. Thus  $\int_{_{0}}^{a}f(t)dt=0$. Therefore, 
$a=0$, that is, $\lim_{n\rightarrow \infty }a_{n}=0$. We claim $%
\{x_{n}\}_{n\geq 1}$ is a Cauchy sequence in $(X,d)$ i.e, 
\begin{equation*}
\forall \varepsilon >0\ \ \ \exists N_{\varepsilon }\in \mathbb{N}\mid
\forall m,n\in \mathbb{N},m>n>N_{\varepsilon }~~~d(x_{m},x_{n})<\varepsilon .
\end{equation*}
Suppose, to the contrary, that $\{{x}_{n}\}$ is not a Cauchy sequence. By
Lemma \ref{goodlem}\ there exists $\varepsilon {\ >}0\ $ for which we can find
subsequences $\{{x}_{n(k)}\}$ and $\{{x}_{m(k)}\}$ of $\{{x}_{n}\}$ with ${n}%
(k){>m}(k){>}k$ such that
\begin{align}
\label{golm}
\epsilon &=\lim_{k\rightarrow \infty }d(x_{m(k)},x_{n(k)})=\lim_{k\rightarrow
\infty }d(x_{m(k)},x_{n(k)+1})=\lim_{k\rightarrow \infty
}d(x_{m(k)+1},x_{n(k)})\notag \\
&=\lim_{k\rightarrow \infty }d(x_{m(k)+1},x_{n(k)+1})
\end{align}
Now by relations (\ref{1}) and (\ref{2}), we have 
\begin{equation*}
\psi (\int_{_{0}}^{d(x_{m_{N_{k}}+1},x_{n_{N_{k}}+1})}f(t)dt)\leq F(\psi
(\int_{_{0}}^{d(x_{m_{N_{k}}},x_{n_{N_{k}}})}f(t)dt),\varphi
(\int_{_{0}}^{d(x_{m_{N_{k}}},x_{n_{N_{k}}})}f(t)dt))
\end{equation*}%
or 
\begin{equation*}
\psi (\int_{_{0}}^{d(x_{m_{N_{k}}+2},x_{n_{N_{k}}+1})}f(t)dt)\leq F(\psi
(\int_{_{0}}^{d(x_{m_{N_{k}+1}},x_{n_{N_{k}}})}f(t)dt),\varphi
(\int_{_{0}}^{d(x_{m_{N_{k}+1}},x_{n_{N_{k}}})}f(t)dt)).
\end{equation*}%
When $k\rightarrow \infty $, we have 
\begin{equation*}
\psi (\int_{_{0}}^{\varepsilon }f(t)dt)\leq F(\psi (\int_{_{0}}^{\varepsilon
}f(t)dt),\varphi (\int_{_{0}}^{\varepsilon }f(t)dt)).
\end{equation*}
So, $\psi (\int_{_{0}}^{\varepsilon }f(t)dt)=0$ or $\varphi(\int_{_{0}}^{\varepsilon }f(t)dt)=0$. Thus  $\int_{_{0}}^{\varepsilon
}f(t)dt=0$ and hence, $\varepsilon =0$, which is a contradiction. So
there exists $k\in \mathbb{N}$ such that for each natural number $N>k$ one
has $d(x_{m_{N}+1},x_{n_{N}+1})<\varepsilon $ and $%
d(x_{m_{N}+2},x_{n_{N}+1})<\varepsilon $. Now we claim that there exist a $%
\delta _{\varepsilon }\in (0,\varepsilon )$ and $N_{\varepsilon }\in \mathbb{%
N}$ such that for each natural number $N>N_{\varepsilon }$, we have 
\begin{equation*}
d(x_{m_{N}+1},x_{n_{N}+1})<\varepsilon -\delta _{\varepsilon }\ \ \text{or}\ \
d(x_{m_{N_{k}}+2},x_{n_{N_{k}}+1})<\varepsilon -\delta _{\varepsilon }.
\end{equation*}%
Suppose that exist a subsequence $\{N_{k}\}_{k\geq 1}\subseteq \mathbb{N}$
such that $d(x_{m_{N_{k}}+1},x_{n_{N_{k}}+1})\rightarrow \varepsilon $ and $%
d(x_{m_{N_{k}}+2},x_{n_{N_{k}}+1})\rightarrow \varepsilon $ as $k\rightarrow
\infty $. Now by relations (\ref{1}) and (\ref{2}), we have 
\begin{equation*}
\psi (\int_{_{0}}^{d(x_{m_{N_{k}}+1},x_{n_{N_{k}}+1})}f(t)dt)\leq F(\psi
(\int_{_{0}}^{d(x_{m_{N_{k}}},x_{n_{N_{k}}})}f(t)dt),\varphi
(\int_{_{0}}^{d(x_{m_{N_{k}}},x_{n_{N_{k}}})}f(t)dt))
\end{equation*}%
or 
\begin{equation*}
\psi (\int_{_{0}}^{d(x_{m_{N_{k}}+2},x_{n_{N_{k}}+1})}f(t)dt)\leq F(\psi
(\int_{_{0}}^{d(x_{m_{N_{k}+1}},x_{n_{N_{k}}})}f(t)dt),\varphi
(\int_{_{0}}^{d(x_{m_{N_{k}+1}},x_{n_{N_{k}}})}f(t)dt)),
\end{equation*}%
which is a contradiction. Now if $d(x_{m_{N}+1},x_{n_{N}+1})<\varepsilon
-\delta _{\varepsilon }$, then 
\begin{equation*}
\varepsilon \leq d(x_{m_{N}},x_{n_{N}})\leq d(x_{m_{N}},x_{m_{N}+1})
\end{equation*}%
\begin{equation*}
+d(x_{m_{N}+1},x_{n_{N}+1})+d(x_{n_{N}+1},x_{n_{N}})
\end{equation*}%
\begin{equation*}
<d(x_{m_{N}},x_{m_{N}+1})+(\varepsilon -\delta _{\varepsilon
})+d(x_{n_{N}},x_{n_{N}+1})
\end{equation*}%
and if $d(x_{m_{N}+2},x_{n_{N}+1})<\varepsilon -\delta _{\varepsilon }$,
then 
\begin{equation*}
\varepsilon \leq d(x_{m_{N}},x_{n_{N}})\leq
d(x_{m_{N}},x_{m_{N}+1})+d(x_{m_{N}+1},x_{m_{N}+2})
\end{equation*}%
\begin{equation*}
+d(x_{m_{N}+2},x_{n_{N}+1})+d(x_{n_{N}+1},x_{n_{N}})
\end{equation*}%
\begin{equation*}
<d(x_{m_{N}},x_{m_{N}+1})+d(x_{m_{N}+1},x_{m_{N}+2})+(\varepsilon -\delta
_{\varepsilon })+d(x_{n_{N}},x_{n_{N}+1}).
\end{equation*}%
So, we have $\varepsilon \leq \varepsilon -\delta _{\varepsilon }$ when $%
N\rightarrow \infty $, which is a contradiction. This proves our claim that $%
\{x_{n}\}_{n\geq 1}$ is a Cauchy sequence in $(X,d)$. Let $%
\lim_{n\rightarrow \infty }x_{n}=x$. By relations (\ref{1}) and (\ref{2}), for each $n\geq 1$
either\newline
\\
(i) 
\begin{align*}
\psi (\int_{_{0}}^{d(Tx_{n},Tx)}f(t)dt)\leq F(\psi
(\int_{_{0}}^{d(x_{n},x)}f(t)dt),\varphi
(\int_{_{0}}^{d(x_{n},x)}f(t)dt))\leq \psi (\int_{_{0}}^{d(x_{n},x)}f(t)dt)
\end{align*}
or\newline
\\
(ii) 
\begin{align*}
\psi (\int_{_{0}}^{d(Tx_{n+1},Tx)}f(t)dt)&\leq F(\psi
(\int_{_{0}}^{d(x_{n+1},x)}f(t)dt),\varphi
(\int_{_{0}}^{d(x_{n+1},x)}f(t)dt))\\
&\leq \psi
(\int_{_{0}}^{d(x_{n+1},x)}f(t)dt)
\end{align*}
holds and then $\int_{_{0}}^{d(Tx_{n},Tx)}f(t)dt\rightarrow 0$ or $%
\int_{_{0}}^{d(Tx_{n+1},Tx)}f(t)dt\rightarrow 0$, when $n\rightarrow \infty $%
. Thus $\lim_{n\rightarrow \infty }d(Tx_{n},Tx)=0$ or $\lim_{n\rightarrow
\infty }d(Tx_{n+1},Tx)=0$. In case (i), since 
\begin{equation*}
d(x,Tx)\leq d(x,Tx_{n})+d(Tx_{n},Tx)=d(x,x_{n+1})+d(Tx_{n},Tx),
\end{equation*}%
we obtain $d(x,Tx)=0$ and so $Tx=x$. We obtain $Tx=x$, similar to cace (i),
from case (ii). Now we shall show that this fixed point is unique. Suppose that
there are two distinct points $a,b\in X$ such that $Ta=a$ and $Tb=b$. Since $%
d(a,b)>0=\alpha d(a,Ta)$, we have the contradiction 
\begin{equation*}
0<\psi (\int_{0}^{d(a,b)}f(t)dt)=\psi (\int_{0}^{d(Ta,Tb)}f(t)dt)\leq F(\psi
(\int_{0}^{d(a,b)}f(t)dt),\varphi (\int_{0}^{d(a,b)}f(t)dt)).
\end{equation*}%
So, $\psi (\int_{0}^{d(a,b)}f(t)dt)=0$ or $\varphi
(\int_{0}^{d(a,b)}f(t)dt)=0$. Thus $\int_{0}^{d(a,b)}f(t)dt=0$. To prove
that $\lim_{n\rightarrow \infty }T^{n}x=a$, let $x$ be arbitrary and $a\in Fix(T)$. Note that since $d(a,T^{n-1}x)\geq 0=\alpha d(a,Ta)$ for every 
$n\in N$, we have 
\begin{eqnarray*}
\psi (\int_{0}^{d(a,T^{n}x)}f(t)dt) &\leq &F(\psi
(\int_{0}^{d(a,T^{n-1}x)}f(t)dt),\varphi (\int_{0}^{d(a,T^{n-1}x)}f(t)dt)). 
\end{eqnarray*}
Letting $n\rightarrow \infty$ in above inequality, we conclude that
\begin{align*}
\psi (\delta ) &\leq F(\psi (\delta ),\varphi (\delta ))\\
&\leq \psi (\delta ).
\end{align*}
So, $\psi (\delta )=0$ or $\varphi (\delta )=0$. Therefore, $\delta =0$,
 thus $\int_{0}^{d(a,T^{n}x)}f(t)dt\rightarrow 0$. Hence, $\lim_{n\rightarrow
\infty }T^{n}x=a$.
\end{proof}

\section{Example and Application}

In this section, we give some remarks and examples which clarify the
connection between our result and the classical ones. As an application, the
existence of a continuous solution for an integral equation is obtained.

\begin{remark}
If in Theorem \ref{2.1}, we take $F(s,t)=\beta s$, $\psi(t)=At$ and $\varphi(t)=Bt$, where $A,B\in 
(0,\infty)$, then Theorem \ref{1.4} is obtained. 
\end{remark}

\begin{remark}
Theorem \ref{2.1} is a generalization of Theorem \ref{1.3}. Letting $f(t)=1$
for each $t\geq 0$, $F(s,t)=\beta s$, $\psi(t)=At$ and $\varphi(t)=Bt$, where $A,B\in 
(0,\infty)$ in Theorem \ref{2.1}, we have 
\begin{align*}
\psi(\int_{0}^{d(Tx,Ty)}f(t)dt)&=Ad(Tx,Ty)\\
&\leq F(\psi(\int_{0}^{d(x,y)}f(t)dt),\varphi(\int_{0}^{d(x,y)}f(t)dt))\\
&=A\beta d(x,y),
\end{align*}
for all $x,y\in X$, i.e. $d(Tx,Ty)\leq \beta d(x,y)$. The converse is not true as we will see in Example %
\ref{e1}.
\end{remark}

\begin{remark}
Theorem \ref{2.1} is a generalization of Theorem \ref{1.2}. The convers is
not true as we will see the Example \ref{e1}.
\end{remark}

\begin{example}
\label{e1} Let $X:\{(0,0),(5,6),(5,4),(0,4)\}\cup \{(n,0):n\in \mathbb{N}%
\}\cup \{(n+12,n+13):n\in \mathbb{N}\}$ and its metric defined by $%
d((x_{1},x_{2}),(y_{1},y_{2}))=\vert x_{1}-y_{1}\vert+\vert x_{2}-y_{2}\vert$. Define 
$F:[0,\infty)^{2}\rightarrow \mathbb{R}$ by $F(s,t)=\frac{1}{2}s$, suppose
$\psi, \varphi : [0,\infty)\rightarrow [0,\infty)$ defined by $\psi(t)=2t$, $\varphi(t)=t$ and define
mapping $T$ on $X$ by 
\begin{equation*}
T((x_{1},x_{2}))=\left\{ 
\begin{array}{ll}
(x_{1},0)  & x_{1}\leqslant x_{2}, \\ 
&  \\ 
(0,x_{2})  & x_{2}<x_{1}. 
\end{array}%
\right.
\end{equation*}%
Then $T$ satisfies the assumptions of Theorem \ref{2.1} with $%
f(t)=t^{t}(1+\ln t)$ for $t>0$, $f(0)=0$, $\alpha =5/12$ $\beta =1/2$ while $%
T$ is not satisfies the assumptions of Theorem \ref{1.2}. First note that
\begin{align*}
\psi(\int_{0}^{d(Tx,Ty)}f(t)dt)\leq F (\psi(\int_{0}^{d(x,y)}f(t)dt),
\varphi(\int_{0}^{d(x,y)}f(t)dt)),
\end{align*}
i.e.,
$\int_{0}^{d(Tx,Ty)}f(t)dt\leq \frac{1}{2}\int_{0}^{d(x,y)}f(t)dt$ if $%
(x,y)\neq ((5,6),(5,4))$ and $(x,y)\neq ((5,4),(5,6)).$ In this context one
has $\int_{0}^{x}f(t)dt=x^{x}$. Let $d(Tx,Ty)=n$ and $d(x,y)=m$. It is clear
that $n<m$ if $(x,y)\neq ((5,6),(5,4))$ and $(x,y)\neq ((5,4),(5,6)).$ Then
we have 
\begin{align*}
\psi(\int_{0}^{d(Tx,Ty)}f(t)dt)&=2\int_{0}^{n}f(t)dt=2n^{n}<\frac{1}{2}(2m^{m})\\
&=\frac{1}{2}\psi(\int_{0}^{m}f(t)dt)\\
&=F(\psi(\int_{0}^{d(x,y)}f(t)dt),\varphi(\int_{0}^{d(x,y)}f(t)dt)),
\end{align*}%
because 
\begin{equation*}
\frac{n^{n}}{m^{m}}=\frac{n^{n}}{m^{n+k}}=(\frac{n}{m})^{n}\frac{1}{m^{k}}<%
\frac{1}{2}.
\end{equation*}%
On the other hands, since $\alpha d((5,6),T(5,4))>5/2>2$ and $\alpha
d((5,4),T(5,6))>25/12>2,$ $T$ satisfies the assumption in Theorem \ref{2.1}.
\end{example}

\begin{remark}
Let $x = (n+12,n+13)$ and $y = (n,0)$. Then in Exampel \ref{e1}, we have $\frac{d(Tx,Ty)}
{d(x,y)} =\frac{n+12}{n+25}$ and so $\sup_{x,y\in X\setminus \{(5,6),(5,4)\}}
\frac{d(Tx,Ty)}{d(x,y)} = 1$. Thus $T$ is not a contraction mapping.
\end{remark}

Let us consider the following integral equation: 
\begin{align}
x(t)=g(t)+\int_{0}^{t}K(s,x(s))ds,\ \ t\in [0,1],
\label{4.4}
\end{align}%
we are going to give existence and uniqueness results for the solution of
the integral equation using Theorem \ref{2.1}.\newline
Let us consider $%
X:=(C([0,1],\Vert .\Vert _{\infty })$.
\begin{theorem}
Consider the integral equation (\ref{4.4}). Suppose\newline
i) $K:[0,1]\times \mathbb{R}^{n}\rightarrow \mathbb{R}^{n}$
and $g:[0,1]\rightarrow \mathbb{R}^{n}$ are continuous;\newline
ii) there exist $\alpha \in (0,\frac{1}{2}]$, $\psi \in \Psi ,\ \varphi \in \Phi _{u},\ F\in \mathcal{C}$ 
such that \newline 
 $\alpha \vert x(t)-g(t)-\int_{0}^{t}K(s,x(s))ds\vert\leq \vert x(t)-y(t)\vert$ implies 
\begin{equation*}
\psi(\int_{_{0}}^{\vert K(t,x(t))-K(t,y(t))\vert}f(\lambda )d\lambda )\leq
F(\psi(\int_{_{0}}^{\vert x(t)-y(t)\vert}f(\lambda )d\lambda),\varphi(\int_{_{0}}^{\vert x(t)-y(t)\vert}f(\lambda )d\lambda)),
\end{equation*}
for all $x,y\in X$ and $f:[0,\infty )\rightarrow (0,\infty )$ is a Lebesgue
integrable mapping which is summable and for each $\varepsilon >0$, $%
\int_{_{0}}^{\varepsilon }f(\lambda )d\lambda >0.$ Then the integral
equation (2), have a unique solution.
\end{theorem}

\begin{proof}
Let $T:X\rightarrow X$, $x\mapsto T(x)$, where 
\begin{equation*}
T(x)(t)=\int_{0}^{t}K(s,x(s)ds+g(t),\ \ \ t\in \lbrack 0,1].
\end{equation*}%
In this way, the integral equation (\ref{4.4}) can be written as $x=T(x)$. We are
going to show that $T$ satisfy the conditions of Theorem \ref{2.1}. Let $%
x,y\in X$ and $\alpha \vert x(t)-g(t)-\int_{0}^{t}K(s,x(s))ds\vert\leq \vert x(t)-y(t)\vert$.
Then 
\begin{equation*}
\alpha \Vert x-Tx\Vert _{\infty }\leq \Vert x-y\Vert _{\infty }
\end{equation*}%
implies 
\begin{align*}
\psi(\int_{0}^{\Vert Tx-Ty\Vert _{\infty }}&f(\lambda )d\lambda)\\
&=\psi(\int_{0}^{\max_{t\in \lbrack 0,1]}\vert Tx(t)-Ty(t)\vert}f(\lambda )d\lambda)\\
&\leq \psi(\int_{_{0}}^{\max_{t\in \lbrack
0,1]}\int_{0}^{t}\vert K(s,x(s))-K(s,y(s))\vert ds}f(\lambda )d\lambda)\\
&\leq \psi(\int_{_{0}}^{\max_{s\in \lbrack 0,1]}\vert K(s,x(s))-K(s,y(s))\vert}f(\lambda
)d\lambda)\\
&\leq F(\psi(\int_{_{0}}^{\max_{s\in \lbrack 0,1]}\vert x(s)-y(s)\vert}f(\lambda )d\lambda),\varphi(\int_{_{0}}^{\max_{s\in \lbrack 0,1]}\vert x(s)-y(s)\vert}f(\lambda )d\lambda))\\
&\leq F(\psi(\int_{_{0}}^{\Vert x-y\Vert _{\infty }}f(\lambda )d\lambda),
\varphi(\int_{_{0}}^{\Vert x-y\Vert _{\infty }}f(\lambda )d\lambda)).
\end{align*}%
Now Theorem \ref{2.1} shows that there exists $x_{0}\in X$ such that $%
Tx_{0}=x_{0}$ and so 
\begin{equation*}
x_{0}(t)=Tx_{0}(t)=\int_{0}^{t}K(s,x(s)ds+g(t).
\end{equation*}
\end{proof}



\begin{thebibliography}{99}

\bibitem {1} M. Abbas, B. Rhoades, Common fixed point theorems for hybrid
pairs of ocasionally weakly compatible mappings satisfying generalized
contractive condition of integral type, Hindawi Publishing Corporation,
Fixed point theory and applications, Volume 2007, Article ID 54101.

\bibitem {2} S.M.A. Aleomraninejad, Sh. Rezapour, N. Shahzad, 
On fixed point generalizations of Suzuki's method, Applied Mathematics Letters,
24 (2011), 1037-1040.

\bibitem {0} S.M.A. Aleomraninejad and M. Shokouhnia,
Some fixed point results of integral type and applications, 
Adv. Fixed Point Theory, 5 (2015), No. 1, 101-109.

\bibitem {16} A.H. Ansari,
Note on \textquotedblright\ $\varphi $ --$\psi $
-contractive type mappings and related fixed point\textquotedblright , The
2nd Regional Conference onMathematics And Applications,Payame Noor
University, 2014, pages 377-380.

\bibitem {3} H. Aydi, A common fixed point result by altering distances
involving a contractive condition of integral type in partial metric spaces,
Demonstratio Mathematica, Vol. XLVI, No. 2, 2013.

\bibitem {17} G.V.R. Babu and P.D. Sailaja,A Fixed Point Theorem of
Generalized Weakly Contractive Maps in Orbitally Complete Metric Spaces, Thai
Journal of Mathematics, Vol. 9 (2011), No. 1, 1-10.

\bibitem {4} S. Banach, Sur les oprations dans les ensembles
abstraits et leur application aux quations intgrales, Fund. Math., 3 (1922), 133-181 (French).

\bibitem {5} A. Branciari, A fixed point theorem for mappings satisfying a
general contractive condition of integral type, Hindawi Publishing
Corpration, Inter. J. Math. Math. Sci., 29 (2002), 531-536.

\bibitem {6} F.S. De Blasi, J. Myjak, S. Reich, A.J. Zaslavski, Generic
existence and approximation of fixed points for nonexpansive set-valued
maps, Set-Valued Var. Anal., 17 (2009), 97-112.

\bibitem {7} S. Dhompongsa, H. Yingtaweesittikul, Fixed point for multivalued
mappings and the metric completeness, Fixed Point Theory and Applications,
2009, 15 pages, Article ID 972395.

\bibitem {8} M. Kikkawa, T. Suzuki, Three fixed point theorems for generalized
contractions with constants in complete metric spaces, Nonlinear Analysis, 69
(2008), 2942-2949.

\bibitem {9} G. Mot, A. Petru\c{s}el, Fixed point theory for a new type of
contractive multivalued operators, Nonlinear Analysis, 70 (2009), 3371-3377.

\bibitem {10} S. Reich, A.J. Zaslavski, Convergence of inexact iterative
schemes for nonexpansive set-valued mappings, Hindawi Publishing Corpration,
Fixed point theory and applications, Vol. 2010, Article ID 518243.

\bibitem {11} S. Reich, A.J. Zaslavski, Existence and approximation of fixed
points for set-valued mappings, Commun. Math. Anal., 8 (2010), 70-78.

\bibitem {12} D.R. Smart, Fixed Point Theorems, Cambridge University Press,
London, 1974.

\bibitem {13} T. Suzuki, A new type of fixed point theorem in metric spaces,
Nonlinear Analysis, 71 (2009), 5313-5317.

\bibitem {14} T. Suzuki, A generalized Banach contraction principle that
characterizes metric completeness, Proc. Amer. Math. Soc., 136 (2008),
1861-1869.

\bibitem {15} P. Vijayaraju, B.E. Rhoades and R. Mohanraj, A fixed point
theorem for a pair of maps satisfying a general contractive condition of
integral type, Hindawi Publishing Corpration, 2005:15 (2005), 2359-2364.

\end{thebibliography}
\end{document}